\documentclass[12pt]{amsart}

  \usepackage{amsfonts,amssymb,amscd,amsmath,enumerate,verbatim,calc,latexsym,pstcol,pst-plot,pst-3d,}
  \usepackage{a4wide}

  \usepackage{boxedminipage}
  \newenvironment{note}[1][Note]
   {\bigskip\begin{center}\begin{boxedminipage}{4.5in}\setlength{\parindent}{1em}\noindent\textbf{#1. }}
   {\end{boxedminipage}\end{center}\bigskip}

  \usepackage{graphicx}

  \def\NZQ{\mathbb}               % the font for N,Z,Q,R,C
  
  \def\QQ{{\NZQ Q}}
  
  \def\RR{{\NZQ R}}
  \def\CC{{\NZQ C}}

  \def\FF{{\NZQ F}}

  \def\frk{\mathfrak}               % font for "Fraktur"

  \def\Pp{{\frk P}}

  \def\mm{{\frk m}}
  
  \def\nn{{\frk n}}
  \def\Pp{{\frk P}}
  \def\Kp{{\frk K}}
  \def\Phi{{\frk N}}
  \def\xb{{\bold x}}

  \def\opn#1#2{\def#1{\operatorname{#2}}} % to make operators
  \opn\chara{char} \opn\length{\ell} \opn\pd{pd} \opn\rk{rk}
  \def\projdim{\mathrm{pd}} \opn\injdim{inj\,dim} \opn\rank{rank}
  \def\apd{\mathrm{apd}} \opn\injdim{inj\,dim} \opn\rank{rank}
  \opn\depth{depth} \opn\grade{grade} \opn\height{height}
  \opn\embdim{emb\,dim} \opn\codim{codim}
  
  \opn\Tr{Tr} \opn\bigrank{big\,rank}
  \opn\superheight{superheight}\opn\lcm{lcm}
  \opn\trdeg{tr\,deg}%\emph{
  \opn\reg{reg} \opn\lreg{lreg} \opn\ini{in} \opn\lpd{lpd}
  \opn\size{size}\opn{\mult}{mult}
  \opn\div{div} \opn\Div{Div} \opn\cl{cl} \opn\Cl{Cl}
  \opn\Spec{Spec} \opn\Supp{Supp} \opn\supp{supp} \opn\Sing{Sing}
  \opn\Ass{Ass} \opn\Min{Min}
  \opn\Ann{Ann} \opn\Rad{Rad} \opn\Soc{Soc}
  \opn\Syz{Syz} \opn\Im{Im} \opn\Ker{Ker} \opn\Coker{Coker}
  \opn\Am{Am} \opn\Hom{Hom} \opn\Tor{Tor} \opn\Ext{Ext}
  \opn\End{End} \opn\Aut{Aut} \opn\id{id}
  
  \opn\nat{nat}
  \opn\pff{pf}%   \pf exists already
  \opn\Pf{Pf} \opn\GL{GL} \opn\SL{SL} \opn\mod{mod} \opn\ord{ord}
  \opn\Gin{Gin}
  \opn\Hilb{Hilb}\opn\adeg{adeg}\opn\std{std}\opn\ip{infpt}
  \opn\Pol{Pol}
  \opn\sat{sat}
  \opn\Var{Var}
  \opn\Gen{Gen}
  
  \opn\aff{aff} \opn\con{conv} \opn\relint{relint} \opn\st{st}
  \opn\lk{lk} \opn\cn{cn} \opn\core{core} \opn\vol{vol}
  \opn\link{link} \opn\star{star}
  \opn\gr{gr}
  
  \def\pot#1#2{#1[\kern-0.28ex[#2]\kern-0.28ex]}

  \opn\dirlim{\underrightarrow{\lim}}
  \opn\inivlim{\underleftarrow{\lim}}
  \let\iso=\cong

  \let\Dirsum=\bigoplus

  \let\to=\rightarrow
  
  \def\Implies{\ifmmode\Longrightarrow \else
        \unskip${}\Longrightarrow{}$\ignorespaces\fi}
  \def\implies{\ifmmode\Rightarrow \else
        \unskip${}\Rightarrow{}$\ignorespaces\fi}
  \def\iff{\ifmmode\Longleftrightarrow \else
        \unskip${}\Longleftrightarrow{}$\ignorespaces\fi}

  \let\:=\colon
  \newtheorem{Theorem}{Theorem}[section]
  \newtheorem{Lemma}[Theorem]{Lemma}
  \newtheorem{Corollary}[Theorem]{Corollary}
  \newtheorem{Proposition}[Theorem]{Proposition}
  \newtheorem{Remark}[Theorem]{Remark}
  
  \newtheorem{Example}[Theorem]{Example}

  \newtheorem{Conjecture}[Theorem]{Conjecture}
  
  \let\epsilon\varepsilon
  \let\phi=\varphi
  \let\kappa=\varkappa
  \def\qed{\ifhmode\textqed\fi
      \ifmmode\ifinner\quad\qedsymbol\else\dispqed\fi\fi}
  \def\textqed{\unskip\nobreak\penalty50
       \hskip2em\hbox{}\nobreak\hfil\qedsymbol
       \parfillskip=0pt \finalhyphendemerits=0}
  \def\dispqed{\rlap{\qquad\qedsymbol}}
  
  \opn\dis{dis}
  \def\pnt{{\raise0.5mm\hbox{\large\bf.}}}
  
  \opn\Lex{Lex}

\begin{document}

  \title{The Betti polynomials of powers of an ideal}

  \author{ J\"urgen Herzog  and Volkmar Welker}

   \address{J\"urgen Herzog, Fachbereich Mathematik, Universit\"at Duisburg-Essen, Campus Essen, 45117
            Essen, Germany} \email{juergen.herzog@uni-essen.de}

   \address{Volkmar Welker, Philipps-Universit\"at Marburg, Fachbereich Mathematik und Informatik,
            35032 Marburg, Germany} \email{welker@mathematik.uni-marburg.de}

  %\dedicatory{}

   \begin{abstract}
      For an ideal $I$ in a regular local ring or a graded ideal $I$ in the polynomial ring we study the limiting
      behavior of $\beta_i(S/I^k)=\dim_K\Tor_i^S(S/\mm,S/I^k)$ as $k$ goes to infinity. By
      Kodiyalam's result it is known that $\beta_i(S/I^k)$ is a polynomial for large $k$. We call
      these polynomials the Kodiyalam polynomials and encode the limiting behavior in their generating
      polynomial. It is shown that the limiting behavior depends only on the coefficients on the
      Kodiyalam polynomials in the highest possible degree. For these we exhibit lower bounds in special
      cases and conjecture that the bounds are valid in general. We also show that the Kodiyalam polynomials have weakly
      descending degrees and identify a situation where the polynomials have all highest possible
      degree.
   \end{abstract}

  \maketitle

  \section{Introduction}

    Let $S$ be either a regular local ring with maximal ideal $\mm$ and residue class field $K$ or a polynomial ring
    over $K$ with maximal graded ideal $\mm$. We assume that $\dim S=n$.  Furthermore, let $I$ be a proper
    (graded) ideal in $S$. In his paper \cite{Kodiyalam1993} Kodiyalam proved that
    \[
       \beta_i(S/I^k)=\dim_K\Tor_i^S(S/\mm,S/I^k)
    \]
    as a function of $k$ is a polynomial function of degree $\leq \ell(I)-1$ for $k\gg 0$. Here $\ell(I)$ denotes the
    {\em analytic spread} of $I$, that is, the Krull-dimension of the fiber $R(I)/\mm R(I)$ of the Rees algebra
    $R(I)=\Dirsum_{k\geq 0}I^kt^k$. It is known and easy to prove that $\height(I)\leq \ell(I)\leq \dim S$.

    We denote by $\Pp_i(I)$ the polynomial with $\Pp_i(I)(k)=\beta_i(S/I^k)$ for $k\gg 0$. and call the polynomials
    $\Pp_0(I), \Pp_1(I),\ldots, \Pp_n(I)$ the {\em Kodiyalam polynomials} of $I$. Note that $\Pp_0(I)=1$.

    It is an immediate consequence of Kodiyalam's result, see  Remark \ref{projdim}, that the projective dimension
    $\projdim(S/I^k)$ of $S/I^k$ stabilizes for $k \gg 0$. Indeed this fact was proved by different means first by
    Brodmann \cite{Brodmann1979}. Note, that Brodmann's result was formulated in terms of the depth rather than the
    projective dimension. We write $\apd(I)$ for $\lim_{k \rightarrow \infty} \projdim(S/I^k)$ and call
    $\apd(I)$ the {\it asymptotic projective dimension} of $I$.

    In this paper we are interested in the limiting behavior of the polynomial $$\Pp(I)(k,t)  = \sum_{i=0}^{\apd(I)}
    \Pp_{i}(I)(k)t^{\apd(I)-i}$$ as $k$ goes to infinity.
    Clearly, at least $\Pp_1(I)(k)$ goes to infinity if $\ell(I) \neq 1$.
    Indeed, in Proposition \ref{first} we show that
    $\ell(I) -1 = \deg \Pp_1(I)\geq \deg \Pp_2(I)\geq \ldots \geq \deg \Pp_{\apd(I)}(I)$.
    In the proof of Proposition \ref{first}, essentially following the ideas by Kodiyalam \cite{Kodiyalam1993},
    we identify $\Pp_i(I)$ as the Hilbert polynomial of the some
    finitely generated module. Therefore, the leading coefficient of $\Pp_i(I)$ is of the form $k_i/d_i!$
    where $d_i=\deg \Pp_i(I)$.  By $\Kp(I)$ we denote $\max \{ i~|~d_i = \ell(I) -1 \}$.
    Note, that the preceding facts imply that $k_i$ is the multiplicity of a finitely generated module.

    We show that the limiting behavior for $k \rightarrow \infty$ of
    $\Pp(I)(k,t)$ is up to convergence rate completely determined by the polynomial
    $\sum_{i=1}^{\Kp(I)} k_i \cdot t^{\apd(I)-i}$.
    More precisely:

     \begin{Theorem} \label{limittheorem}
       Let $I$ be a (graded) ideal in $S$ such that $\ell(I) \geq 2$. Let $\alpha_1, \ldots, \alpha_{\Kp(I)-1}$
       be the roots of the polynomial $\displaystyle{\sum_{i=1}^{\Kp(I)} k_i \cdot t^{\apd(I)-i}}$.
       Then for $1 \leq i \leq \apd(I)$ there are sequences
       $(\gamma_k^{(i)})_{k \geq 1}$ of complex numbers, such
       that after suitable numbering:
       \begin{itemize}
         \item[(i)] $\displaystyle{\prod_{i=1}^{\apd(I)}} (t - \gamma_k^{(i)}) = \Pp(I)(k,t)$ for all $k \geq 1$.
         \item[(ii)] $\gamma_k^{(i)} \rightarrow \alpha_i$, $1 \leq i \leq \apd(I) -1$, for $k \rightarrow \infty$.
         \item[(iii)] $\gamma_k^{(\apd -1)} = \alpha_{\apd-1} = -1$, for all $k \geq 1$.
         \item[(iv)] $\gamma_k^{(\apd(I))} \in \RR$ for $k \gg 0$ and $\gamma_k^{(\apd(I))} \rightarrow -\infty$ for
             $k \rightarrow \infty$.
       \end{itemize}
    \end{Theorem}

    The assumption $\ell(I) \geq 2$ is equivalent to saying that $I$ is not a principal ideal. Clearly, for principal
    ideals $I$, each power $I^k$ is principal and $\beta_0 = \beta_1 = 1$, $\beta_i = 0$ for $i \geq 2$ which
    is a trivial situation for our purposes.

    Theorem \ref{limittheorem} focuses our interest on the number $\Kp(I)$ and the multiplicities $k_i$ for
    $1 \leq i \leq \Kp(I)$.
    Note, that in Theorem \ref{limittheorem}
    the number of $\alpha_i$ equal to $0$ is $\apd(I) - \Kp(I)$ and that for $1 \leq i \leq \Kp(I)$ we
    have $$\frac{k_i}{k_1} = \lim_{k\to \infty} \frac{\beta_i(S/I^k)}{\beta_1(S/I^k)}.$$
    The following two are our main results.

    \begin{Theorem} \label{main}
       Suppose that $\ell(I)=n$. Then $\Kp(I) = n$, in particular, $\deg \Pp_i(I)=n-1$ for $i=1,\ldots,n$.
    \end{Theorem}

    \begin{Theorem} \label{binom}
       Suppose that $R(I)/\mm R(I)$ is a domain and $R(I)_{\mm R(I)}$
       is Cohen--Macaulay.
       Then $k_i/k_1\geq {\Kp(I)-1\choose i-1}$ for $i=1,\ldots,\Kp(I)$.
       Moreover, equality holds if and only if $R(I)_{\mm R(I)}$ is a complete intersection.
    \end{Theorem}

    As a first corollary we get that the inequality from Theorem \ref{binom} holds if $\ell(I) = n$ and
    $R(I)/\mm R(I)$ is a domain. Observe that $R(I)/\mm R(I)$ is always a domain if $I$ is a graded ideal in the
    polynomial ring generated by elements of the same degree. From this remark and Theorem \ref{binom} we deduce in
    a second corollary that equality holds for Artinian monomial ideals generated in a single degree with linear relations.

    Based on experimental data we conjecture that the inequality from Theorem \ref{binom} holds in general.

    \begin{Conjecture} \label{conjecture}
      Let $I\subset S$ be a (graded) ideal. Then
      \[
        \lim_{k\to \infty} \frac{\beta_i(S/I^k)}{\beta_1(S/I^k)} = \frac{k_i}{k_1} \geq {\Kp(I) -1\choose i-1}
                     \quad\text{for}\quad i=1,\ldots,\Kp(I).
      \]
    \end{Conjecture}

    We note that the condition $\frac{k_i}{k_1} \geq {\Kp(I) -1\choose i-1}$
    from Conjecture \ref{conjecture} is
    satisfied whenever the polynomial $\sum_{i=1}^{\Kp(I)} k_{i} \cdot t^{\apd(I)-i}$ has only real roots
    (see \cite[Observation 3.4]{BellSkandera2007}).
    Indeed, we know of no example for which the polynomial is not real rooted. But we consider our evidence too weak 
    for a conjecture. Indeed, we see in Remark \ref{minusone} that for $\ell(I) \geq 2$ we have that $-1$ is always a 
    root. In addition, in all example we tried experimentally $\ell(I)$ was small and there were only very few roots
    other than $-1$.

    For the class of monomial ideals it is an interesting question which of the invariants defined for $I$ in the
    introduction can depend on the characteristic of the field. The fact that $\ell(I)$ is independent of the 
    field is an immediate consequence of a convex geometric description in \cite{BiviaAusina2003} 
    (see also \cite[Corollary 4.10]{Singla2007}). On the other hand for the invariants 
    $\apd(I)$, $\Pp_i(I)$ for some $i > 1$, $\Kp(I)$ and then $k_i$ for some $i > 1$ we do not know of a proof 
    nor a counterexample. In general counterexamples are hard to find, due to the fact that only small powers of 
    monomial ideals can be treated with the existing computer algebra systems.

  \section{The Kodiyalam polynomials of an ideal}

    Before we come to a more subtle analysis of the polynomials $\Pp_i(I)(k)$ we state a simple
    consequence of the fact that $\beta_i(S/I^k)=\Pp_i(I)(k)$ for $k \gg0$.
    As mentioned in the introduction the conclusion was first shown by Brodmann
    \cite{Brodmann1979} in terms of depth.

    \begin{Remark} \label{projdim}
      The projective dimension $\projdim(S/I^k)$ stabilizes for $k \gg 0$.
    \end{Remark}
    \begin{proof}
      Let $q=\max\{i\:\; \Pp_i(I)\neq 0\}$, and let $k_0$ be an
      integer such that $\Pp_i(I)(k)=\beta_i(S/I^k)$ for all $k\geq k_0$. Since $\Pp_i(I)(k)$ has only finitely many
      zeroes, we may also assume that $\Pp_q(I)(k)\neq 0$ for all $k\geq k_0$. Then $\projdim (S/I^k)=q$ for all
      $k\geq k_0$.
    \end{proof}

    For a polynomial $P$ we set $\deg P=-\infty$ if is the zero polynomial. Using this convention we get.

    \begin{Proposition} \label{first}
      $\ell(I)-1=\deg \Pp_1(I)\geq \deg \Pp_2(I)\geq \ldots \geq \deg \Pp_n(I)$.
    \end{Proposition}

    \begin{proof}
       For $i\geq 1$ we have
       \[
          \beta_i(S/I^k)=\beta_{i-1}(I^k)=\dim_K\Tor_{i-1}^S(S/\mm, I^k)=\dim_KH_{i-1}({\xb};I^k).
       \]
       Here $H_i({\xb};I^k)$ is the $i$th Koszul homology of $I^k$ with respect to ${\xb}=x_1,\ldots,x_n$,
       where ${\xb}$ is a regular system of parameters if $S$ is a regular local ring, and is the sequence of indeterminates
       in case $S$ is a polynomial ring.

       Observe that $H_i({\xb};R(I))$ is a graded $H_0({\xb};R(I))$-module. Thus by
       $H_0({\xb};R(I))= R(I)/\mm R(I)$ it is a graded
       $R(I)/\mm R(I)$-module.
       Since $H_i({\xb};R(I))_k=H_i({\xb};I^k)$ for all $k$, we see that $\Pp_i(I)$ is the
       Hilbert polynomial of $H_{i-1}({\xb};R(I))$ for $i \geq 1$. Thus the degree of $\Pp_i(I)$ is the Krull
       dimension of $H_{i-1}({\xb};R(I))$ minus $1$. In particular, $\deg \Pp_1(I)=\dim R(I)/\mm R(I)-1=\ell(I)-1$.

       In order to prove the inequalities $\deg \Pp_{i+1}(I)\leq \deg \Pp_{i}(I)$, it remains to show that
       $\dim H_i({\xb};R(I))\leq \dim H_{i-1}({\xb};R(I))$ for all $i\geq 1$. To see this, let
       $P\in \Supp H_i({\xb};R(I))$. Then $\mm R(I)\subset P$ and  $H_i({\xb};R(I)_P)=H_i({\xb};R(I))_P
       \neq 0$. Rigidity of the Koszul homology (see \cite[Exercise 1.6.31]{BrunsHerzog1998}) implies that
       $H_{i-1}({\xb};R(I))_P=H_{i-1}({\xb};R(I)_P)\neq 0$. Thus
       $\Supp(H_i({\xb};R(I))\subset \Supp H_{i-1}({\xb};R(I))$, which yields the desired inequality for the
       dimensions.
   \end{proof}

   We give a first example which shows that there are cases where the inequalities in Proposition \ref{first} are
   indeed equalities.

   \begin{Example} \label{firstexample}
    {\em    Let $I=(x^3, x^2-yz, y^4+xz^3, xy-z^2)\subset S= K[x,y,z]$. The ideal $I$ is $(x,y,z)$-primary, so that
       $\ell(I)=3$ and $\projdim S/I^k=3$ for all $k$. It follows form Theorem~\ref{main} that $\deg \Pp_i(I)=2$ for $i=1,2,3$.
       A calculation with CoCoA indicates that $\Pp_1(I)(k)=(k+1)^2$, $\Pp_2(I)(t)=(\frac{5}{2}k+\frac{7}{2})k$ and
       $\Pp_3(I)(k)=\frac{3}{2}k(k+1)$.
       So here we have $2 = \ell(I) -1 = \deg \Pp_1(I) = \deg \Pp_2(I) = \deg \Pp_3(I)$.
       More precisely, $k_1 = 6$, $k_2 = 15$ and $k_3 = 21$.}
    \end{Example}

    The second example shows that even for monomial ideals the inequalities from Proposition \ref{first} can be strict.

    \begin{Example} \label{secondexample}
      {\em
       Consider the monomial ideal  $$I=(a^6,a^5b,ab^5,b^6,a^4b^4c,a^4b^4d,a^4e^2f^3)$$ in $\QQ[a,b,c,d,e,f]$.
       Then $\Pp_1(I)(k) = 3k^2+4k-7$,
       $\Pp_2(I)(k) = 6k^2+3k-7$, $\Pp_3(I)(k) = 3k^2-k+5$, $\Pp_4(I)(k) = 5$, $\Pp_5(I)(k) = 1$ and $\Pp_6(I)(k) = 0$.
       Thus $\deg \Pp_i(I)=2$ for $i=1,2,3$, while
       $\Pp_4(I)$ and $\Pp_5(I)$ are of degree $0$,  and $\Pp_6(I)$ is the zero polynomial. In particular, $\Kp(I) = 3$.}
    \end{Example}

    In the light of Proposition \ref{first} and Examples \ref{firstexample} and \ref{secondexample},
    Theorem \ref{main} provides sufficient conditions for extremal behavior of $\Kp(I)$.

    \begin{proof}[Proof of Theorem \ref{main}]
       It has been shown by Brodmann \cite{Brodmann1979} that $\projdim S/I^k\geq \ell(I)$ for $k\gg 0$. Thus our assumptions imply
       that $\projdim S/I^k=n$ for $k\gg 0$. Therefore,  $\Pp_n(I)\neq 0$ and $\deg \Pp_n(I) \geq 0$.

       We will show that $\deg \Pp_n(I)=n-1$,
       equivalently, that $\dim H_{n-1}({\xb};R(I))=n$. Then the assertion of the theorem follows from
       Proposition~\ref{first}.

       Notice that
       \[
          H_{n-1}({\xb};R(I))_k=H_{n-1}({\xb};I^k)\iso H_n({\xb};S/I^k)\iso (I^k:_S \mm)/I^k.
       \]
       Hence as an $R(I)/\mm R(I)$-module $$H_{n-1}({\xb};R(I)) \cong \Dirsum_{k\geq 0}((I^k:_S \mm)/I^k)t^k.$$
       By the following inclusion of $R(I)/\mm R(I)$-modules
       \[
          (R(I):_{S[t]}\mm R(I))/R(I)\subset \Dirsum_{k\geq 0}\Big((I^k:_S \mm)/I^k\Big)t^k
       \]
       it suffices to prove that the dimension of $(R(I):_{S[t]} \mm R(I))/R(I)$ is equal to $n$.

       We may assume that $n>1$, because otherwise the theorem is trivially true. Let $L$ be the quotient field of $R(I)$.
       Then $L$ is also the quotient field of $S[t]$.

       \smallskip

       \noindent {\sf Claim 1:}
       \[
          (R(I):_{S[t]} \mm R(I)) = (R(I):_{L} \mm R(I)).
       \]

       \smallskip

       \noindent $\triangleleft$ {\sf Proof of Claim 1:}
       Let $f\in L$ with $f\mm R(I)\subset  R(I)$, and let $N=S\setminus\{0\}$.  Localizing with respect to $N$,
       we obtain that $\mm R(I)_N=R(I)_N= L_0[t]$, where $L_0$ is the quotient field of $S$. Note, that $L$ is also the quotient field
       of $L_0[t]$. Thus $f\mm R(I)\subset  R(I)$ yields $f L_0[t]\subset L_0[t]$,  which implies that $f \in L_0[t]$.
       Let $f=\sum_{i=0}^rf_it^i$ with $f_i\in L_0$.
       Then, since $f\mm\in R(I)$, it follows that $f_i\mm\in I^i\subset S$, and hence  we see that $f_i\in S$ because $\dim S>1$.
       Therefore, we conclude $f\in S[t]$, as desired. $\triangleright$

       By Claim 1 we have reduced the assertion follows if we show that $\dim (R(I):_L \mm R(I))/R(I) = n$.
       By assumption,  $\dim R(I)/\mm R(I)=n$. Since $R(I)$ is Catenarian and since $\dim R(I)=n+1$, it follows that
       $\height \mm R(I)=1$. Let $P$ be a prime ideal in $R(I)$ with $\height P=1$ and set $T=R(I)_P$. Then $T$ is a one
       dimensional local domain with quotient field $L$ and $(R(I):_L \mm R(I))_P= (T:_L \mm T)$.

       \smallskip

       \noindent {\sf Claim 2:} $(T:_L \mm T) \neq T$.

       \smallskip

       Claim 2 implies that $((R(I):_L \mm R(I))/R(I))_P\neq 0$, so that $P$
       is in the support of the  module $(R(I):_L \mm R(I))/R(I)$. Consequently, $\dim (R(I):_L \mm R(I))/R(I)\geq \dim R(I)/P=
       \dim R(I)-1=n$. Since the reversed inequality is trivially true, the desired equality follows.

       \noindent $\triangleleft$ {\sf Proof of Claim 2:} Suppose that $(T:_L \mm T)=T$. Then
       $(T:_L \mm^2 T)=(T:_L \mm T):_L \mm T=(T:_L \mm T)=T$. By induction on $k$, one gets that
       $(T:_L \mm^k T)=T$ for all $k$. Let $x\in \mm T$, $x\neq 0$. Since $\dim T=1$, there exists an integer $k$ such that
       $\mm ^k T\subset xT$. Hence $x^{-1}\mm^k T\subset T$, so that $x^{-1}\in (T:_L \mm^k T)=T$. This is a contradiction,
       since $x$ is a non-unit in $T$, because $x\in \mm T$. $\triangleright$
   \end{proof}

   In case $I$ is $\mm$-primary the consequence of Theorem \ref{main} was first proved using different means 
   in \cite{FallaLaBarbieraStagliano2009}. In this case the result also follows by the subsequent short argument that
   was provided to the first author by S. Goto. Let $G(I) = \Dirsum_{k \geq 0} I^k/I^{k+1}$ be the associated graded ring
   of the $\mm$-primary ideal $I$ and assume $n > 0$. Then choose a prime $P \in \Ass G(I)$ with $\dim G(I)/P = n$. Since
   $I$ is $\mm$-primary, $\mm G(I)$ is a nilpotent ideal in $G(I)$. Hence $\mm G(I) \subseteq P$ and  $G(I)/P
   \subseteq (0):_{G(I)} \mm$. From that it follows that the $S$-length of 
   $H_{n-1}({\xb};R(I))_k \cong (I^k:_S \mm)/I^k$ is a polynomial in $k$ of degree $n-1$. By Proposition \ref{first} the
   assertion follows.
 
   We now turn our attention to the multiplicities $k_i$ for $1 \leq i \leq \Kp(I)$.

   \begin{Remark} \label{minusone}
    If $\ell(I) \geq 2$, then
     \[
        \sum_{i=1}^{\Kp(I)} (-1)^ik_i = 0.
     \]
   \end{Remark}
   \begin{proof}
     Since $\displaystyle{\sum_{i=0}^n} (-1)^i \beta_i(S/I^k)=0$ for all $k \geq 1$, it follows that
     $\sum_{i=0}^n(-1)^{i+1}\Pp_i(I)(k) =0$. All terms in the alternating sum are polynomials for $k \gg 0$. Therefore,
     for any $k$-power the alternating sum of the coefficients cancels. Now by
     $\ell(I) \geq 2$,  the maximal degree $\ell(I)-1 > 0 = \deg \Pp_0(I)(k)$ is achieved for $\Pp_i(I)(k)$, $1 \leq i \leq \Kp(I)$.
     This implies the assertion.
   \end{proof}

   If one looks at the actual values of the $k_i$ in Example \ref{firstexample} one observes that
   $\Kp(I) = 2$ and $k_i/k_1\geq {2 \choose i-1}$,  and in Example~\ref{secondexample} we have
   $\Kp(I) = 3$ and $k_i/k_1 = {3 \choose i-1}$.  Theorem \ref{binom} provides
   conditions under which inequalities of that type hold. Before we can proceed to the proof of Theorem
   \ref{binom} we need the following lemma.

   \begin{Lemma} \label{lower}
      Let $P$ be a prime ideal of height $h$ in a regular local ring $(R,\nn, K)$. Then
      \begin{eqnarray} \label{equality}
         \dim_K \Tor_i^R(K, R/P)\geq {h\choose i} \quad \text{for}\quad i=1,\ldots,h.
      \end{eqnarray}
      Equality holds if and only $P$ is generated by a regular sequence.
   \end{Lemma}

   \begin{proof}
      Let $\FF$ be a minimal free $R$-resolution of $R/P$. The ring $R_P$ is a regular local ring of dimension $h$, and the
      localization $\FF_P$ is a free resolution of the residue class field $R_P/PR_P$. Since $PR_P$ is generated by a regular 
      sequence of length $h$,  we see that
      \[
         \dim_K\Tor_{i}^R(K, R/P)=\rank_R F_{i}= \rank_{R_P} (F_{i})_P\geq {h\choose i-1}.
      \]

      On the other hand, if $P$ is generated by a regular sequence, then the Koszul complex of this sequence provides a minimal free
      $R$-resolution of $R/P$, and equality holds in (\ref{equality}).

      Conversely, suppose we have equality in (\ref{equality}). Then $\dim_K \Tor_1^R(K, R/P)=h$, which implies that $P$ is generated by
      $h$ elements. Since $h$ is the height of $P$, these elements form a regular sequence
   \end{proof}

   \begin{proof}[Proof of Theorem \ref{binom}]
       By the proof of Theorem \ref{limittheorem} the multiplicity of the $R(I)\mm R(I)$-module $H_{i-1}({\xb};R(I))$ is $k_i$. In particular,
       $k_1$ is the multiplicity of $R(I)\mm R(I) = H_{0}({\xb};R(I))$. Hence by
       \cite[Corollary 4.6.9]{BrunsHerzog1998} it follows that
       \[
          k_i = k_1 \cdot \rank H_{i-1}({\xb};R(I))\quad  \text{for}\quad i=1,\ldots,\Kp(I).
       \]
       Set $T=R(I)_{\mm R(I)}$ and denote by $W$ the residue class field of the local ring $T$. Then for $i=1,\ldots,\Kp(I)$ the rank of
       $H_{i-1}({\xb};R(I))$ is the vector space dimension of the $W$-vector space $H_{i-1}({\xb};T)$. Since ${\xb}$ is a system of
       generators of $\mm R(I)$, the
       numbers $\dim_WH_{i-1}({\xb};T)$ have the following interpretation: suppose $I$ is generated by $f_1,\ldots,f_m$. Let
       $A=S[y_1,\ldots,y_m]$ be the polynomial over $S$  in the variables $y_i$.  Let $J$ denote the kernel of the canonical, surjective
       $S$-algebra homomorphism $\varphi\:  A\to R(I)$  with $y_i\mapsto f_i$ for $i=1,\ldots,m$, and set  $P=(J,\mm)$. Then $P$ is a prime
       ideal and $B=A_P$ is a regular local ring. The algebra homomorphism $\varphi$ induces then a surjective homomorphism $B\to T$ of
       local rings, and it follows that
       \begin{eqnarray*} \label{q-1}
          \dim_W H_{i-1}({\xb};T)=\dim_W\Tor_{i-1}^B(W, T)\quad \text{for}\quad i=1,\ldots,\Kp(I).
       \end{eqnarray*}
       In particular, $\projdim_B T=\Kp(I)-1$, since $H_{\Kp(I)-1}({\xb};T)\neq 0$, but $H_{i-1}({\xb};T)=0$ for $i>\Kp(I)$.
       Let $H$ be the kernel of $B\to T$. Then $H$ is a prime ideal with
       \[
          \height  H=\dim B-\dim T= \dim B-\depth T=\projdim_B T=\Kp(I)-1.
       \]
       Here we have employed the assumption that $T$ is Cohen--Macaulay.

       The assertions of the theorem now follow from Lemma \ref{lower}
       applied to the prime ideal $H$ and the regular local ring $B$.
   \end{proof}

   \begin{Corollary} \label{limit}
      Suppose that $R(I)/\mm R(I)$ is a domain and that  $\ell(I)=n$. Then
      \[
         \lim_{k\to \infty} \frac{\beta_i(S/I^k)}{\beta_1(S/I^k)} = \frac{k_i}{k_1}
         \geq {n-1\choose i-1} \quad \text{for}\quad i=1,\ldots,n.
      \]
   \end{Corollary}

   \begin{proof}
      Since $\ell(I)=n$, it follows that  $P=\mm R(I)$  is a prime ideal of height $1$. Therefore, $R(I)_P$ is a one dimensional
      local domain and hence Cohen--Macaulay. Thus we may apply Theorem~\ref{binom} and obtain
      \begin{eqnarray*}
         \lim_{k\to \infty} \frac{\beta_i(S/I^k)}{\beta_1(S/I^k)}
                                 & = & \lim_{k\to\infty} \frac{\Pp_i(I)(k)}{\Pp_1(I)(k)} \\
                                 & = & \lim_{k\to\infty}\displaystyle{\frac{\frac{k_i}{(n-1)!}k^{n-1}+\cdots}
                                                                      {\frac{k_1}{(n-1)!}k^{n-1}+\cdots}}\\
                                 & = & \frac{k_i}{k_1}.
      \end{eqnarray*}
   \end{proof}

   In the next result we describe a situation in which the hypotheses of Theorem \ref{binom} for the 
   equality conclusion are satisfied.

   \begin{Corollary} \label{satisfied}
      Let $I\subset S$ be a monomial ideal generated in a single degree with $\dim S/I=0$.  Suppose  that $I$ has linear relations.
      Then
      \[
         \lim_{k\to \infty} \frac{\beta_i(S/I^k)}{\beta_1(S/I^k)} = \frac{k_i}{k_1} = {n-1\choose i-1}  \quad \text{for}\quad i=1,\ldots,n.
      \]
   \end{Corollary}

   \begin{proof}
      Let $I=(u_1,\ldots,u_m)$ be the monomial generators of $I$, each of degree $d$. Since they are all of same degree,
      it follows that $R(I)/\mm R(I)\iso K[u_1,\ldots,u_m]$. In particular, $R(I)/\mm R(I)$ is a domain. We denote the prime ideal
      $\mm R(I)$ by $P$, and show that $R(I)_P$ is a discrete valuation ring. Then it follows that $\height P=1$, so that $\ell(I)=n$,
      and Theorem~\ref{binom} yields the desired equations.

      In order to prove that $R(I)_P$ is a discrete valuation ring, it suffices to show that $PR(I)_P$ is generated by one element.
      Let $\xb = (x_1, \ldots, x_n)$ be a regular system of parameters in case $S$ is a regular local ring and the sequence of 
      indeterminates in case $S$ is a polynomial ring. Observe, that $(x_1,\ldots,x_n)R(I)_P = PR(I)_P$. We will show that each 
      $x_i$ differs from $x_1$ only a by unit, form which the desired conclusion will follow.

      Since $\dim S/I=0$, we have that $x_i^d\in I$ for $i=1,\ldots, n$. Let $F$ be the free $S$-module with basis $e_1,\ldots,e_m$
      and let $\epsilon \:\ F\to I$ the $S$-module epimorphism with $\epsilon(e_i)=u_i$ for $i=1,\ldots,m$. Let $i$ be an integer
      with $1<i\leq m$. Since $I$ has linear relations, the relation $x_i^de_1=x_1^de_i$ can be expressed as a multihomogeneous linear
      combination of linear relations, namely
      \[
          x_i^de_1-x_1^de_i=\sum_jv_jr_j,
      \]
      with $v_j$  monomials and  relations  $r_j=x_{j_k}e_{j_k}-x_{j_l}e_{j_l}$, and where the multidegree of each summand is  equal
      to the multidegree $x_1^dx_i^d$. It follows that $\{x_{j_k},x_{j_l}\}=\{x_1,x_i\}$ for all $j$. We choose one of the relations
      $r_j$ in this sum, and may assume that $r_j=x_1e_{j_k}-x_ie_{j_l}$. This relation gives rise to the equation
      $x_1(u_{j_k}t)=x_i(u_{j_l}t)$ in the Rees algebra $R(I)$. Since the elements $u_it$ do not belong to $P$, they become units in
      $R(I)_P$. Thus the preceding equation shows that $x_1$ and $x_i$ only differ by a unit $R(I)_P$, as desired.
   \end{proof}

   We note that the conclusion of Corollary \ref{satisfied} is valid in many
   cases that do not satisfy its assumptions.

   \begin{Example} \label{thirdexample}
    {\em   Let $I = (xy,vw,xz)$ then $\ell(I) = 3 = \apd(I)$ and
      $\Pp_1(I)(k) = \frac{1}{2}k^2+\frac{3}{2}+1$, $\Pp_2(I)(k) = k^2+2k$
      and $\Pp_3(I)(k) = \frac{1}{2} k^2+\frac{1}{2} k$. Thus
      $\Kp(I) = 3$ and $k_1 = 1 = {\Kp(I) -1 \choose 0}$, $k_2 = 2
      = { \Kp(I)-1 \choose 2}$ and $k_3 = 1 = {\Kp(I) -1 \choose 2}$.
      But $I$ does not have linear relations by $\beta_{2,4}(S/I) = 1$.}
   \end{Example}

\section{Roots of Polynomials}

    Before we can prove Theorem \ref{limittheorem} we need a technical lemma.
    A similar lemma, albeit for polynomials with a different structure,
    appears in \cite{BW} in another context.  

    \begin{Lemma} \label{technicallemma}
      Let $(f_k(t))_{k \geq 1}$ be a sequence of real polynomials of degree
      $\leq q-1$ and $f(t)$ a non-zero real polynomial of degree $q-1$. Assume that
      all $(f_k(t))_{k \geq 1}$  and $f(t)$ have non-negative
      coefficients. Let $\ell$ be a natural number such that:
      \begin{itemize}
         \item[$\triangleright$] $\lim_{k \rightarrow \infty} f_k(t)/k^\ell =
                                  0$, where the limit is taken in $\RR^{q}$.
      \end{itemize}
      Let $\alpha_1, \ldots, \alpha_{q-1}$ be the roots of $f(t)$.
      Then there are sequences $(\gamma_i^{(k)})_{k \geq 1}$, $1 \leq i \leq q$ of complex numbers such that:
      \begin{itemize}
        \item[(i)] $\displaystyle{\prod_{i = 1}^{q} (t-\gamma_i^{(k)}) = f_k(t) + k^\ell f(t) +t^q}$.
        \item[(ii)] $\gamma_i^{(k)} \rightarrow \alpha_i$, $1 \leq i \leq q-1$, for $k \rightarrow \infty$.
        \item[(iii)] $\gamma_q^{(k)}$ is real for $k \gg 0$ and $\gamma_{q}^{(k)} \rightarrow -\infty$ for
               $k \rightarrow \infty$.
        \end{itemize}
    \end{Lemma}

    \begin{proof}
       Consider a zero $\alpha_i$ of the polynomial $f(t)$. Let
       $\epsilon > 0$ be such that $f(t) \neq 0$ for $0 < |t - \alpha_i| < 2\epsilon$.
       Set $G_\epsilon^i = \{ t~|~|t - \alpha_i| \leq \epsilon \}$.
       We claim that for large enough $k$ the polynomial $f_k(t) + k^\ell f(t) +t^q$ has
       a zero in $G_\epsilon^i$.
       Assume not. Then we can find arbitrarily large $k$ for which $g_k(t) := f_k(t) + k^\ell f(t) +t^q$
       does not vanish in $G_\epsilon^i$. Then $1/g_k(t)$ is holomorphic inside $G_\epsilon^î$.
       By the maximum principle the maximum of $1/g_k(t)$ on $G_\epsilon^i$ is obtained
       on the boundary of $G_\epsilon^i$.
       In particular, this implies that there is a $t_0$ such that $|t_0 - \alpha_i| = \epsilon$ and
       $|1/g_k(t_0)| > |1/g_k(\alpha_i)|$. Hence $|g_k(\alpha_i)| > |g_k(t_0)|$. Thus

       \begin{eqnarray*}
         |f_k(\alpha_i) + \alpha_i^q| & > & |f_k(t_0) + k^\ell f(t_0) + t_0^q|
       \end{eqnarray*}

       This implies
       \begin{eqnarray*}
         1/k^\ell |f_k(\alpha_i)| + 1/k^\ell |\alpha_i|^q & > & |1/k^\ell f_k(t_0) + f(t_0) + 1/k^\ell t_0^q|
       \end{eqnarray*}

       Since by assumption the left hand side converges to $0$ for $k \rightarrow \infty$
       and the right hand side to $|f(t_0)| > 0$ we obtain a contradiction.
       Hence there is a zero of $f_k(t) + k^\ell f(t) +t^q$ in $G_\epsilon^i$ for large $k$.

       Now we choose $\epsilon$ small enough so that the $G_\epsilon^i$, $1 \leq i \leq q-1$,
       are pairwise disjoint.
       In this situation and for large enough $k$ we
       denote by $\gamma_i^{(k)}$ the zero of $f_k(t) + k^\ell f(t) +t^q$ in the disk $G_\epsilon^i$
       around $\alpha_i$ with radius $\epsilon$.
       Then as $\epsilon$ goes to $0$ the root $\gamma_i^{(k)}$ converges to $\alpha_i$, $1 \leq i \leq q-1$.
       Since for $k \rightarrow \infty$ at least one coefficient of $g_k(t)$ goes to
       infinity there must be at least one root with modulus going to infinity. We call this
       root $\gamma_q^{(k)}$.

       The argumentation so far shows that for each distinct root of $f(t)$ there is a sequence
       of roots of $f_k(t) + k^\ell f(t) + t^q$ converging to the root. We are left with
       studying multiple roots. Assume $\alpha$ is an $r$-fold root of $f(t)$ for
       some $r \geq 2$. In this case $\alpha$ is also a root of
       $k^\ell \frac{\partial^i}{\partial^i t} f(t)$ for $0 \leq i \leq r-1$. Consider the
       polynomial
       $$ \frac{1}{q}
          \left( 
            k^\ell \frac{\partial}{\partial t} f(t) +
            \frac{\partial}{\partial t} f_k(t)+
            q(q-1)\cdots (q-i) t^{q-1} \right).$$
       By induction on $r$ we obtain that this polynomial has $r-1$ roots converging to $\alpha$ as $k$ 
       goes to infinity. 
       Now the assertion follows by \cite[Theorem 3.2.4]{RahmanSchmiesser2002}.

       Since by assumption at least one of the coefficients of $f_k(t) + k^\ell f(t) + t^q$ is
       unbounded and there are $q-1$ bounded roots it follows that there must be a $q$-th root
       that is unbounded. Since $k^\ell f(t) + f_k(t)+t^q$ has real coefficients all
       roots in $\CC \setminus \RR$ come in conjugate pairs. Since there is a unique unbounded
       root it follows that the root is real for large enough $k$. By the property that
       $f_k(t) + k^\ell f(t) + t^q$ has only non-negative coefficients it follows that
       all real roots are non-positive, hence the unbounded roots must go to $-\infty$ as
       $k \rightarrow \infty$. 
    \end{proof}

    \begin{proof}[Proof of Theorem \ref{limittheorem}]
       The assertion follows directly from Lemma \ref{technicallemma} and
       Remark \ref{minusone} if we set $q = \apd(I)$, $f(t) =
       \frac{1}{(\ell(I)-1)!} \sum_{i=1}^{\Kp(I)} k_i t^{\apd(I)-i}$ and
       $f_k(t) = \Pp(I)(k,t) - f(t) -t^{\apd(I)}$.
    \end{proof}

    A sequence $a_0 , \ldots , a_q$ of real numbers is called log-concave if $a_i^2 \geq a_{i-1} a_{i+1}$
    for $1 \leq i \leq q-1$. We say that a non-necessarily log-concave sequence
    $a_0 , \ldots , a_q$ is strictly log-concave at $i$ if  $a_i^2 > a_{i-1} a_{i+1}$.
    Log-concavity of a sequence of strictly positive numbers $a_0 , \ldots , a_q$ implies that the
    sequence is unimodal, i.e. there is an $i$ such that $a_0 \leq \cdots \leq a_i \geq \cdots \geq a_q$.
    This property is of interest in enumerative combinatorics and combinatorial commutative algebra. 
    In the sequel we want to exhibit some facts that allow to deduce partial or full unimodality of
    the sequence $\beta_0(S/I^k) , \ldots, \beta_{\apd(I)} (S/I^k)$ for large $k$.
  
    The next remark identifies situations when we can expect strict log-concavity. The part (i) is a trivial
    consequence of the definition and part (ii) is a well know fact about real rooted polynomials (see
    for example \cite{BellSkandera2007} and the references therein).

    \begin{Remark}
       \begin{itemize}
       {\em   \item[(i)]
           If $a_0, \ldots, a_q$ is a sequence of positive real numbers that is log-concave then
           there are numbers $0 \leq j_1 \leq j_2 \leq q$ such $a_0 < \cdots < a_{j_1} = \cdots = a_{j_2} > \cdots > a_q$.
           In particular, $a_0, \ldots, a_q$ is strictly log-concave at $i$
           for $1 \leq i \leq j_1$ and $j_2 \leq i \leq q-1$.
         \item[(ii)]
           If $a_0 + a_1t +\cdots + a_qt^q \in \RR[t]$ has only real roots then $a_0, \ldots, a_q$ is log-concave.}
       \end{itemize}
    \end{Remark}

    \begin{Corollary} \label{realrootlimit}
       Let $I$ be a (graded) ideal in $S$. Assume that the coefficient series of
       $\sum_{i=1}^{\Kp(I)} k_{i} \cdot t^{\apd(I)-i}$
       is strictly log-concave at $1 \leq i-1,i,i+1 \leq \apd(I) -2$.
       Then for large $k$ the sequence
       $\beta_0(S/I^k), \beta_1(S/I^k), \ldots, \beta_{\apd(I)} (S/I^k)$
       is strictly log-concave at $i$.
    \end{Corollary}
    \begin{proof}
       Using the notation from Theorem \ref{limittheorem} we set
       $$b_k(t) = \frac{1}{(t-\gamma_k^{(\apd(I))})} \Pp(I)(k,t)$$
       and $q = \apd(I)$.
       Then $b_k(t)$ has roots converging to the roots of $\sum_{i=1}^{\Kp(I)} k_{i} \cdot t^{\apd(I)-i}$. Thus up to a constant factor
       the coefficients of
       $b_k(t)$ converge to the coefficients of $\sum_{i=1}^{\Kp(I)} k_i t^{q-i}$. Since the coefficients are continuous in terms of
       roots this implies that the coefficient sequence of $b_k(t)$ is strictly log-concave for large $k$ at $i-1,i$ and $i+1$.
       Now $\Pp(I)(k,t) = \sum_{i=0}^q \Pp_i(I)(k) t^{q-k}$ is obtained from $b_k(t)$ by multiplication with $(t-\gamma_k^{(q)})$.
       Set $\gamma := -\gamma_k^{(q)}$ and write $b_k(t) = c_0 + \cdots + c_{q-2}t^{q-2} +c_{q-1}t^{q-1}$, where
       $c_{q-1} = 1$. If $k$ is large enough and we set $c_{-1} = c_{q} = 0$ then
       $\beta_{q-i}(S/I^k) = \gamma c_i + c_{i-1}$ for $0 \leq i \leq q$.
       Hence strict log-concavity at $i-1,i$ and $i+1$ for large $k$ implies::
       \begin{eqnarray*}
         \beta_{q-i}(S/I^k)^2-\beta_{q-i-1}(S/I^k)\cdot \beta_{q-i+1}(S/I^k)
                                     & = & ( \gamma c_i + c_{i-1})^2- \\
                                     &   & ( \gamma c_{i-1} + c_{i-2})( \gamma c_{i+1} + c_i) \\
                                     & = & \gamma^2( c_i^2-c_{i-1}c_{i+1}) + \\
                                     &   & \gamma (c_{i-1}c_i-c_{i-2}c_{i+1}) + c_{i-1}^2- c_{i-2}c_i \\
                                     & > & \gamma (  c_{i-1}c_i-c_{i-2}c_{i+1})
       \end{eqnarray*}

       Multiplying $c_{i-1}c_i-c_{i-2}c_{i+1}$ by $c_{i-1}c_i$ we obtain $c_{i-1}^2c_i^2-c_{i-2}c_{i-1}c_ic_{i+1}$.
       Again from strict log-concavity we know that $c_{i-1}^2 > c_{i-2}c_i$ and $c_i^2 > c_{i-1}c_{i+1}$. Since the
       coefficients of $b_k(t)$ are positive as they are up to a constant close to the coefficients of
       $\sum_{i=1}^{\Kp(I)} k_{i} \cdot t^{q-i}$
       it follows that $c_{i-1}^2c_i^2-c_{i-2}c_{i-1}c_ic_{i+1} > 0$ and hence $c_{i-1}c_i-c_{i-2}c_{i+1} > 0$.
    \end{proof}

    \begin{figure}[htbp]
      \centering
      \fbox{
        \includegraphics[width=0.7\textwidth]{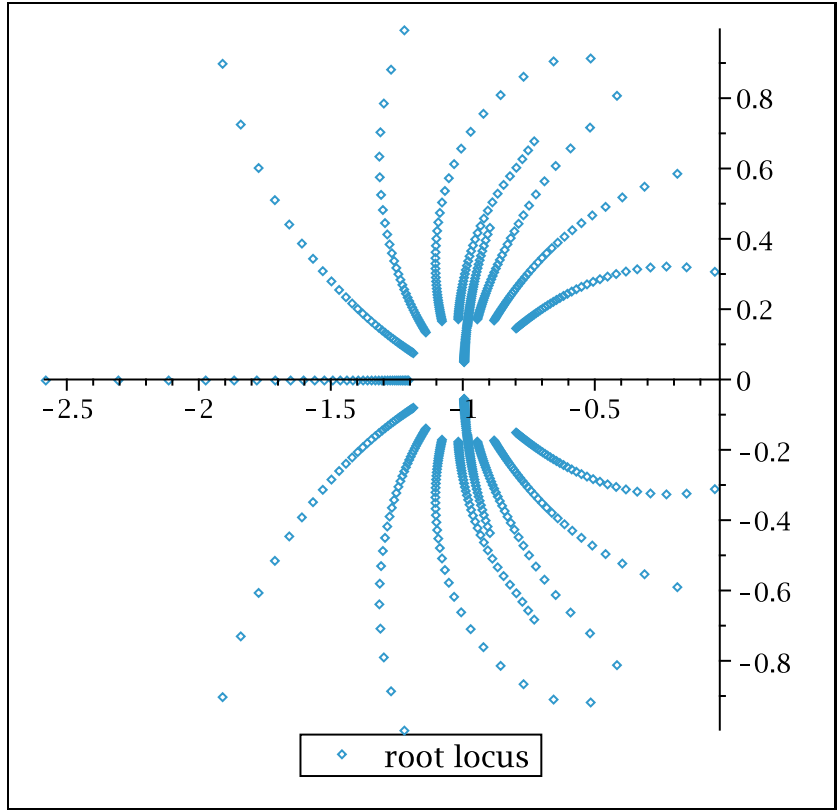}
      }
      \caption{Root loci for Example \ref{maximal} and parameters $n = 20$, $k \in \{1, \ldots, 40\}$}
      \label{roots}
    \end{figure}

    \begin{Example} \label{maximal}
     {\em  Let $I$ be generated by a regular sequence of length $n$. By using the Eagon-Northcott complex we see 
       that $\beta_i(S/I^k) = {k+n-1 \choose n-i}{k-2+i \choose i-1}$ for $1 \leq i \leq \apd(I) = n$.
       Thus
       $$\sum_{i=0}^q \Pp_i(I)(k) t^{n-k} = t^n + \sum_{i=1}^{n} {k+n-1 \choose n-i}{k-2+i \choose i-1} t^{n-i} .$$
      In particular,
      \begin{eqnarray*}
         k_i & = & \frac{(n-1)!}{(n-i)!(i-1)!} \\
      \end{eqnarray*}
      and therefore
      \begin{eqnarray*}
         \sum_{i)=1}^n k_i t^{n-i} & = & \sum_{i=1}^n \frac{(n-1)!}{(n-i)!(i-1)!} t^{n-i} \\
                  & = & \sum_{i=0}^{n-1} {n-1 \choose i} t^{n-1-i} \\
                 & = & \frac{1}{(n-1)!} (1+t)^{n-1}
      \end{eqnarray*}
      Indeed this calculation is predicted by Corollary \ref{satisfied} when $I$ is the maximal (graded)
      ideal in a polynomial ring. The calculation implies that 
      all $\alpha_i$ from Theorem \ref{limittheorem} are equal to $-1$ and the coefficient series is
      the sequence of binomial coefficients which is strictly log-concave.
      Hence Corollary \ref{realrootlimit} applies. Thus for large $k$ the sequence $\beta_0(S/I^k),
      \ldots, \beta_{n} (S/I^k)$ is strictly log-concave and hence unimodal.
      Clearly, this consequences of Corollary \ref{realrootlimit} can also be easily checked by inspection
      of the sequence $\beta_0(S/I^k) = 1$, $\beta_i(S/\mm^k) = {k+n-1 \choose n-i}{k-2+i \choose i-1}$,
      $1 \leq i \leq n$ in this case. This example also shows that the fact that all roots of 
      $\sum_{i)=1}^{\apd(I)} k_i t^{\apd(I)-i}$ are real does not force the
      roots of $\sum_{i=0}^n \Pp_i(I)(k) t^{n-i}$ to be real for large $k$. Indeed, one can check that
      no root except for the two roots forced by Theorem \ref{limittheorem} and depending on the parity
      of $n$ one additional root of $\sum_{i=0}^q \Pp_i(I)(k) t^{q-k}$ are real.
      In Figure \ref{roots} we have depicted the roots for $n = 20$ and $k$ from $1$ to $40$ in this
      example with the imaginary axis being vertical and the real axis being horizontal. Indeed, the real root
      going to $-\infty$ is only seen for small $k$ as it leaves the axis range already for small values of $k$.
      One easily recognizes the root curves converging to $-1$ in conjugate pairs. }
   \end{Example}

   Following the same argumentation as in Example \ref{maximal} we deduce from 
   Corollary \ref{satisfied} and Corollary \ref{realrootlimit} 
   the last result of this paper.

   \begin{Corollary} \label{last}
      Let $I\subset S$ be a monomial ideal generated in a single degree with $\dim S/I=0$.  Suppose  that $I$ has linear relations.
      Then for large $k$ the sequence $\beta_0(S/I^k),
      \ldots, \beta_{n} (S/I^k)$ is strictly log-concave and hence strictly unimodal.
   \end{Corollary} 

   We do not know any ideal $I$ for which the conclusion of Corollary \ref{last} does not hold. But we do not
   see enough evidence to formulate a conjecture.


\begin{thebibliography}{xxx}
     \bibitem{BellSkandera2007} J. Bell, M. Skandera,
        Multicomplexes and polynomials with real zeros,
        Discrete Math. {\bf 307} 668-682 (2007).
     \bibitem{BW} F. Brenti, V. Welker, $f$-vectors of barycentric subdivisions, Math. Z. {\bf 259} 
        849-865 (2008).
     \bibitem{Brodmann1979} M. Brodmann, The asymptotic nature of the analytic spread,
        Math. Proc. Camb. Philos. Soc. {\bf 86} 35-39 (1979).
     \bibitem{BrunsHerzog1998} W. Bruns, J. Herzog,
        {\it Cohen-Macaulay Rings}. Rev. ed.,
        Cambridge Studies in Advanced Mathematics. {\bf 39}. Cambridge: Cambridge University Press
        (1998).
     \bibitem{BiviaAusina2003} C. Bivi{\`a}-Ausina, The analytic spread of monomial ideals, 
        Comm. Algebra {\bf 31} 3487-3496 (2003).
     \bibitem{FallaLaBarbieraStagliano2009}
         C. Falla, M. La Barbiera, P.I. Staglinan\^o , Betti numbers of powers of ideals,
         Le Matematiche, Vol. {\bf LXIII}, Fasc. {\bf II} 191-195 (2008).
     \bibitem{Kodiyalam1993} V. Kodiyalam,
        Homological invariants of powers of an ideal,
        Proc. Am. Math. Soc. {\bf 118} 757-764 (1993).
     \bibitem{RahmanSchmiesser2002} Q.I. Rahman, G. Schmiesser, Analytic Theory of Polynomials,
        London Math. Society Monographs, New Series {\bf 28}, Oxford, Oxford University Press
        (2002). 
     \bibitem{Singla2007} P. Singla, Minimal monomial reductions and the reduced fiber ring of an
        extremal ideal, Ill. J. Math. {\bf 51} 1085-1102 (2007).
  \end{thebibliography}
\end{document}